\documentclass[reqno,oneside,11pt]{amsart}

\usepackage[a4paper, hmargin={2.8cm, 2.8cm}, vmargin={2.5cm, 2.5cm}]{geometry}  
\usepackage[danish,english]{babel} 
\usepackage[latin1]{inputenc} 
\usepackage[T1]{fontenc} 
\usepackage{lmodern}
\usepackage{amsmath,amsthm,amssymb,amsfonts,mathrsfs,latexsym} 
\usepackage{graphicx} 
\usepackage{verbatim} 
\usepackage[all]{xy} 
\usepackage[pagebackref, colorlinks, linkcolor=red, citecolor=blue, urlcolor=blue, hypertexnames=true]{hyperref}  
\usepackage{multirow}
\usepackage{color}

\makeatletter
\newtheorem*{rep@theorem}{\rep@title}
\newcommand{\newreptheorem}[2]{%
\newenvironment{rep#1}[1]{%
 \def\rep@title{#2 \ref{##1}}%
 \begin{rep@theorem}}%
 {\end{rep@theorem}}}
\makeatother

\newtheorem{thm}{Theorem}[section]
\newtheorem{thmx}{Theorem}
\newtheorem{corx}[thmx]{Corollary}
\newreptheorem{thm}{Theorem}
\newtheorem{lem}[thm]{Lemma}
\newtheorem{prop}[thm]{Proposition}
\newtheorem{cor}[thm]{Corollary}
\newtheorem*{thm*}{Theorem}
\newtheorem*{problem*}{Problem}

\newtheorem*{claim*}{Claim}
\theoremstyle{definition}
\newtheorem{defi}[thm]{Definition}
\newtheorem{exem}[thm]{Example}

\newtheorem{rem}[thm]{Remark}
\newtheorem*{conjecture}{Conjecture}
\newcommand{\claimmark}{\hfill$\lozenge$}

\newcommand{\mr}{\mathrm}

\newcommand{\norm}[1] {\| #1 \|}

\newcommand{\C}{\mathbb{C}}
\newcommand{\N}{\mathbb{N}}

\newcommand{\ra}{\rightarrow}
\renewcommand{\epsilon}{\varepsilon}
\renewcommand{\phi}{\varphi}

\renewcommand{\hat}{\widehat}

\renewcommand{\bar}{\overline}

\DeclareMathOperator{\supp}{supp}

\renewcommand{\Im}{\mr{Im}}
\newcommand{\Tube}{\mr{Tube}}
\newcommand{\ip}[2]{\langle {#1} , {#2} \rangle}

\DeclareFontFamily{U}{mathx}{\hyphenchar\font45}
\DeclareFontShape{U}{mathx}{m}{n}{
      <5> <6> <7> <8> <9> <10>
      <10.95> <12> <14.4> <17.28> <20.74> <24.88>
      mathx10
      }{}
\DeclareSymbolFont{mathx}{U}{mathx}{m}{n}
\DeclareFontSubstitution{U}{mathx}{m}{n}
\DeclareMathAccent{\widecheck}{0}{mathx}{"71}
\DeclareMathAccent{\wideparen}{0}{mathx}{"75}

\numberwithin{equation}{section}

\begin{document}
\selectlanguage{english} 


\title[Exactness of Locally Compact Groups]{\texorpdfstring{Exactness of Locally Compact Groups}{Exactness of Locally Compact Groups}}

\author{Jacek Brodzki$^{1}$}
\address{School of Mathematics, University of Southampton,
\newline Southampton, SO17 1BJ, England}
\email{J.Brodzki@soton.ac.uk}

\author{Chris Cave$^{2}$}
\address{Department of Mathematical Sciences, University of Copenhagen,
\newline Universitetsparken 5, DK-2100 Copenhagen \O, Denmark}
\email{chris.cave@math.ku.dk}

\author{Kang Li$^{3}$}
\address{Department of Mathematical Sciences, University of Copenhagen,
\newline Universitetsparken 5, DK-2100 Copenhagen \O, Denmark}
\email{kang.li@math.ku.dk}


\thanks{{$^{1}$ $^{2}$} Supported by the EPSRC grant EP/I016945/1.}
\thanks{{$^{2}$ $^{3}$} Supported by the Danish National Research Foundation through the Centre for Symmetry and Deformation (DNRF92).}
\thanks{{$^{1}$} Supported by the EPSRC grant EP/N014189/1}
\thanks{{$^{3}$} Supported by the Danish Council for Independent Research (DFF-5051-00037).}

\begin{abstract}
We give some new characterizations of exactness for locally compact second countable groups. In particular, we prove that a locally compact second countable group is exact if and only if it admits a topologically amenable action on a compact Hausdorff space. This answers an open question by Anantharaman-Delaroche.
\end{abstract}

\maketitle
\parskip 4pt

 \section{Introduction}
 In their study of the  continuity of fibrewise reduced crossed product $C^*$-bundles
Kirchberg and Wassermann  \cite{MR1721796}  introduced a new class of groups,  called exact groups,  defined by  the following property.  We say that  a locally compact group $G$ is \emph{exact} if the operation of taking reduced crossed products by $G$ preserves  short exact sequences of $G$-$C^*$-algebras.

It is an immediate consequence of this definition (and more details will be given later) that the reduced group $C^*$-algebra $C^*_r(G)$ of an exact group $G$ is an exact $C^*$-algebra in the sense that  the operation of taking the minimal tensor product with $C^*_r(G)$ preserves short exact sequences of $C^*$-algebras. In the same paper, Kirchberg and Wassermann proved the converse assertion for all discrete groups \cite[Theorem 5.2]{MR1721796}. Thus, a discrete group $G$ is exact if and only if its reduced $C^*$-algebra $C_r^*(G)$ is an exact $C^*$-algebra. This result sparked a lot of interest focused on identifying properties of discrete groups that are equivalent to exactness. It was proved by Anantharaman-Delaroche \cite[Theorem 5.8, Theorem 7.2]{MR1926869} and independently by Ozawa \cite[Theorem 3]{MR1763912} that for a discrete group $G$, $C_r^*(G)$ is exact if and only if the uniform Roe algebra $C^*_u(G)$ is nuclear if and only if $G$ acts amenably on a compact Hausdorff space. The latter property is called amenability at infinity.

It is a natural question if a similar relation between exactness and amenability at infinity can be established for general locally compact groups. Here, however, the situation is  more complicated. 
Topological amenability of actions of locally compact groups were studied in depth by  Anantharaman-Delaroche in \cite{MR1926869}. In particular, she proved that a locally compact group is exact whenever it is amenable at infinity \cite[Theorem 7.2]{MR1926869}. 
In the same article Anantharaman-Delaroche also proved that if $G$ is a locally compact group with property (W) such that $C_r^*(G)$ is exact, then $G$ is amenable at infinity \cite[Theorem 7.3]{MR1926869}. Property (W) is weaker than Paterson's inner amenability \cite[\S 2.35]{MR961261}. In particular, every discrete group has property (W). Anantharaman-Delaroche further observed  that property (W) is equivalent to amenability for almost connected groups. 
However, since every almost connected group is amenable at infinity  \cite[Proposition 3.3]{MR1926869}, we see that locally compact groups can be amenable at infinity without having property (W). Based on these facts, Anantharaman-Delaroche left open the following problem: 
\begin{problem*}[{{\cite[Problem 9.3]{MR1926869}}}]
For any locally compact group $G$, does the exactness of $G$ imply that $G$ is amenable at infinity?
\end{problem*}

In this paper, we settle this  problem for the class of locally compact second countable groups:

\begin{thmx}[see Theorem \ref{thm:amenable at infinity and exactness}] \label{thm:intro main theorem}
  Let $G$ be a locally compact second countable group. Then the following conditions are equivalent:
  \begin{enumerate}
   \item $G$ is amenable at infinity.
   \item $G$ is exact.
   \item \label{thm: exact sequence}The sequence
   \[
    0 \to C_0(G) \rtimes_{L,r} G \to C_b^{lu}(G) \rtimes_{L,r} G \to C_b^{lu}(G) / C_0(G) \rtimes_{L,r} G \to 0
   \]
   is exact.
   \item \label{thm: maximal and reduced}$C_b^{lu}(G) / C_0(G)\rtimes_{L}G\cong C_b^{lu}(G) / C_0(G)\rtimes_{L,r}G$ canonically.
   \item $C_b^{lu}(G) \rtimes_{L,r} G$ is nuclear.
  \end{enumerate}
 \end{thmx}
 In the class of discrete countable groups the picture is more complete. Kirchberg and Wassermann proved that for a discrete countable group $G$, exactness of the reduced group $C^*$-algebra $C^*_r(G)$ is equivalent to $G$ being exact \cite[Theorem 5.2]{MR1721796}. 
 
 In the class of discrete groups Anantharaman-Delaroche proved that exactness of $C^*_r(G)$ is equivalent to $G$ being amenable at infinity \cite[Theorem 7.3]{MR1926869} and the nuclearity of the uniform Roe algebra $\ell^{\infty}(G) \rtimes_r G$ \cite[Theorem 5.8]{MR1926869}. This was also proved independently by Ozawa \cite[Theorem 3]{MR1763912} (see also the work of Guentner and Kaminker in \cite{Guentner2002, MR1876896}). In fact in \cite{MR1926869} this equivalence was proved to hold in the class of locally compact groups with property (W).
 
 Higson and Roe in \cite[Theorem 3.3]{MR1739727} proved that Yu's property $A$ for a discrete countable group $G$, introduced in \cite[Definition 2.1]{MR1728880}, is equivalent to $G$ being amenable at infinity.  More recently, Roe and Willett proved that a discrete metric space with bounded geometry has property $A$ if and only if every ghost operator is compact \cite[Theorem 1.3]{MR3146831}. As a consequence, a discrete countable group $G$ is amenable at infinity if and only if the sequence in \eqref{thm: exact sequence} in Theorem \ref{thm:intro main theorem} is exact \cite[Corollary 5.1]{MR3146831}. We are unsure if the equivalence of \eqref{thm: maximal and reduced} in Theorem \ref{thm:intro main theorem} with the other properties in the discrete case has been written down explicitly but it is known to experts and is an easy consequence of \cite[Corollary 5.1]{MR3146831}.
  
In order to prove this main theorem, we will work within the framework of coarse geometry. This is because that it has long been known that the metric space underlying a finitely generated discrete group has property $A$ if and only if the group is amenable at infinity \cite[Theorem 3.3]{MR1739727}. Property $A$ was first introduced in \cite[Definition 2.1]{MR1728880} by Yu on discrete metric spaces and Roe generalized property $A$ to proper metric spaces in \cite[Definition 2.1]{MR2115671}. Recently, the third named author together with Deprez proved in \cite[Corollary 2.9]{MR3420532} that a locally compact second countable group $G$ is amenable at infinity if and only if the metric space $(G,d)$ has Roe's property $A$ with respect to any proper left-invariant metric $d$ that implements the topology on $G$. It was already observed by Roe in \cite[Lemma 2.2]{MR2115671} that if $(G,d)$ does not have Roe's property $A$, then there exists a discrete metric subspace $Z$ without Yu's property $A$. It follows from a recent result in \cite[Theorem 1.3]{MR3146831} by Roe and Willett that there exists a non-compact ghost operator $T$ in the uniform Roe algebra $C_u^*(Z)$ of $Z$. Up to isomorphism, we provide in Propositions \ref{mainprop} and \ref{mainprop2} an embedding $\Phi: C_u^*(Z) \ra C_b^{lu}(G) \rtimes_{L,r} G$ such that $\Phi(T)$ becomes an obstruction to the exactness of the sequence in Theorem \ref{thm:intro main theorem}. This   proves the implication $(3)\Rightarrow(1)$ in the theorem and answers the open question raised by Anantharaman-Delaroche.

Kirchberg and Wassermann in \cite[Section 6]{MR1721796} remarked that they were unable to adapt their proof to show that an arbitrary nondiscrete group with exact reduced group $C^*$-algebra is exact. Anantharaman-Delaroche posed this as a question in \cite[Problem 9.3]{MR1926869} for the class of locally compact groups. As far as we can tell, Theorem \ref{thm:intro main theorem} has no bearing on this problem.
 
 Another open question is the following: if $G$ acts on a compact Hausdorff space $X$ and one knows $C(X) \rtimes G = C(X) \rtimes_r G$, does the reduced crossed product functor $-\rtimes_r G$ preserve short exact sequences? As far as we know, this question is open even in the discrete case.

It is well-known that every locally compact second countable group $G$ admits a proper left-invariant metric $d$ that generates the topology on $G$ \cite{MR0348037} \cite[Theorem 4.5]{HP}. Moreover, property $A$ on locally compact second countable groups is a coarsely invariant property and it is independent of the choices of the proper left-invariant metric $d$ (see Section 2 below). As a corollary of Theorem \ref{thm:intro main theorem} and \cite[Corollary 2.9]{MR3420532}, we obtain the following:
 \begin{corx}
 If $G$ and $H$ are coarsely equivalent locally compact second countable groups, then $G$ is exact if and only if $H$ is exact. 
 \end{corx}

Another interesting consequence of the main theorem relates to the well-known result by Higson  \cite[Theorem 1.1]{MR1779613} that a countable discrete group, which is amenable at infinity, satisfies the strong Novikov conjecture:
 
\begin{conjecture}[The Strong Novikov Conjecture]
Let $G$ be a locally compact second countable group. The Baum--Connes assembly map $$\mu_A:K_*^{\text{top}}(G,A)\rightarrow K_*(A\rtimes_r G)$$ is split-injective for every separable $G$-$C^*$-algebra $A$.
\end{conjecture}
The strong Novikov conjecture for countable discrete groups implies the Novikov conjecture on homotopy invariance of higher signatures \cite[Theorem 7.11]{MR1292018}. Later on, Chabert, Echterhoff and Oyono-Oyono were able to show that the strong Novikov conjecture is still true for locally compact second countable groups that are amenable at infinity \cite[Theorem 1.9]{MR2100669}. Hence, our main theorem provides the following corollary.
  \begin{corx}
 If $G$ is a locally compact second countable exact group, then the strong Novikov conjecture holds for $G$, i.e. the Baum--Connes assembly map $$\mu_A:K_*^{\text{top}}(G,A)\rightarrow K_*(A\rtimes_r G)$$ is split-injective for every separable $G$-$C^*$-algebra $A$.
 \end{corx}

Guentner and Kaminker in \cite[Theorem 3.1]{MR1876896} \cite[Theorem 2]{Guentner2002} and Ozawa in \cite[Theorem 3]{MR1763912} observed that a finitely generated exact group is coarsely embeddable into a Hilbert space. This result was also proved independently by Yu in \cite[Theorem 2.2]{MR1728880}. Gromov's construction \cite{MR1978492, AD} and Osajda's recent construction \cite[Corollary 3.3]{Osa} provide finitely generated groups that are not coarsely embeddable into a Hilbert space. Hence, these groups can not be exact. Osajda also constructed in \cite[Theorem 6.2]{Osa} the first example of a finitely generated non-exact group that is coarsely embeddable into a Hilbert space. However,
the strong Novikov conjecture is known to hold for all discrete groups which admit a coarse embedding into a Hilbert space \cite[Corollary 1.2]{MR1728880} \cite[Theorem 5.5]{MR1905840}. Recently, Deprez and the third named author were able to generalize this result from discrete groups to arbitrary locally compact second countable groups with the same property \cite[Corollary 2.10]{MR3420532}.


We end this introduction with a completely bounded Schur multiplier characterization for locally compact exact groups. The following theorem extends \cite[Theorem~6.1]{MR2945214}, which proved the same statement for discrete groups. However, the proof of \cite[Theorem~6.1]{MR2945214} relies heavily on the discreteness of the groups. For instance, it used the fact from \cite[Theorem 5.2]{MR1721796} that if a reduced discrete group $C^*$-algebra is exact, then the discrete group itself is exact. 

 \begin{thmx}[see Theorem \ref{wa-a}]\label{cb_exact}
Let $G$ be a locally compact second countable group. The following conditions are equivalent:
\begin{itemize}
\item[(1)] The group $G$ is exact.
\item[(2)] There is a constant $C>0$ such that for any compact subset $K\subseteq G$ and $\epsilon>0$, there exist a compact subset $L\subseteq G$ and a (continuous) Schur multiplier $k:G\times G\rightarrow \C$ with $\norm{k}_S\leq C$ such that $\supp k\subseteq \Tube(L)$ and 
\[
\sup\{|k(s,t)-1|: (s,t)\in \Tube(K)\} <\epsilon.
\]
\end{itemize}
 \end{thmx} 
 
 Weak amenability was first introduced by Cowling and Haagerup in \cite{MR996553} and it was proved in \cite[Theorem 2.1]{MR1220905} that discrete weakly amenable groups are exact (see also \cite[Theorem 12.4.4]{MR2391387}). Theorem \ref{cb_exact} has the following immediate consequence:
   \begin{corx}
  Let $G$ be a locally compact second countable group. If $G$ is weakly amenable, then $G$ is exact. 
   \end{corx}
 {\bf Acknowledgments}. We would like to thank Vadim Alekseev, Martin Finn-Sell, Uffe Haagerup and Rufus Willett for helpful discussions on the subject. Finally, thanks to S\o{}ren Knudby and Sven Raum for a careful reading of our first draft and the referee for their useful suggestions.

  \section{Preliminaries} \label{sec:preliminaries}
 In this section we introduce some notions in coarse geometry, which will be used throughout the article. In particular, we will introduce the notion of metric lattices in a locally compact second countable group and explain the connection between Yu's property $A$ on the metric lattices and property $A$ on the locally compact second countable group. We will end this section with two operator algebraic characterizations of Yu's property $A$ on uniformly locally finite discrete metric spaces in terms of the uniform Roe algebra and its ghost ideal.
 
  
\begin{defi}
Let $(X,d_X)$ and $(Y,d_Y)$ be two metric spaces. We say a map $f:X\rightarrow Y$ is \emph{coarse} if the inverse image under $f$ of any bounded subset in $Y$ is bounded in $X$ and if for all $R>0$, there exists $S>0$ such that 
\[
 d_X(x,x')\leq R \Longrightarrow d_Y(f(x),f(x'))\leq S, \ \text{for all $x,x' \in X$}.
\]
Two coarse maps $f, g:X\rightarrow Y$ are \emph{close} if $d_Y(f(x),g(x))$ is bounded on $X$.
\end{defi}
\begin{defi}
We say that two metric spaces $X$ and $Y$ are \emph{coarsely equivalent} if there exist two coarse maps $f:X\rightarrow Y$ and $g:Y\rightarrow X$ such that $g\circ f$ and $f\circ g$ are close to the identity maps on $X$ and $Y$, respectively. 
\end{defi}

\begin{defi}
A metric space $(Z,d)$ is \emph{uniformly locally finite} if $\sup_{z\in Z}|B(z,S)|< \infty$ for all $S>0$, where $B(z,S)$ denotes the closed ball $\{x\in Z:d(z,x)\leq S\}$.
\end{defi}
A uniformly locally finite metric space is necessarily discrete and countable. However, the uniformly local finiteness is not coarsely invariant even on uniformly discrete spaces. Here a metric space $(Z,d)$ is uniformly discrete if there exists $\delta > 0$ such that $d(z,w) > \delta$ for all distinct $z,w \in Z$.
\begin{exem}[{{\cite[Example 3.4]{HP}}}]
Consider the triple $(D_n,d_n,x_n)$, where $D_n$ is the discrete space with $n$ points, $d_n$ is the discrete metric on $D_n$ and each $x_n$ is a fixed element in $D_n$. Let $Z=\sqcup_{n\in \N}D_n$ equipped with the following metric $d$:
\begin{align*}
d(z,y)=d_{j(z)}(z,x_{j(z)})+|j(z)-j(y)|+d_{j(y)}(y,x_{j(y)}),
\end{align*}
where $j(x)=n$ if and only if $x\in D_n$. So $(Z,d)$ is a proper uniformly discrete space that is not uniformly locally finite but is coarsely equivalent to $\N$. In particular, $(Z,d)$ has bounded geometry (see definition below). The metric space $(Z,d)$ is not uniformly locally finite because $|B(x_n,1)|\geq n$ for all $n \in \N$.
\end{exem}

In order to make it into a coarse invariant, we instead consider the following metric notion, which was first introduced by Roe. 
\begin{defi}[{{\cite[Definition 2.3(ii)]{MR1399087}}}]
A metric space $(X,d_X)$ has \emph{bounded geometry} if it is coarsely equivalent to a uniformly locally finite (discrete) metric space $(Z,d_Z)$.
\end{defi}
In fact, we can choose the metric space $(Z,d_Z)$ in the above definition to be a lattice of $(X,d_X)$ in the following sense:

\begin{defi}
Let $(X,d)$ be a metric space. We say that a uniformly discrete subspace $Z\subseteq X$ is a \emph{metric lattice}, if there is $R>0$ such that $X=\bigcup_{z\in Z}B(z,R)$. 
\end{defi}
Note that the inclusion map $Z\subseteq X$ is a coarse equivalence and every metric space always contains metric lattices by Zorn's lemma. Moreover, a metric space $(X,d)$ has bounded geometry if and only if it contains a uniformly locally finite metric lattice $(Z,d)$ \cite[Proposition~3.D.15]{Cornulier2016}.

  

Property A on discrete metric spaces was first introduced by Yu in \cite{MR1728880}. Later on, Roe generalized property $A$ to proper metric spaces with bounded geometry in \cite[Definition 2.1]{MR2115671}. Here we give an equivalent definition of property $A$ in terms of positive type
kernels.

  \begin{defi}[{{\cite[Proposition 2.3]{MR2115671}} }]
    Let $(X,d)$ be a proper metric space with bounded geometry. We say that $(X,d)$ has \emph{property $A$} if for any $R>0$ and $\varepsilon > 0$, there exist $S > 0$ and a continuous positive type kernel $k \colon X \times X \to \C$ such that
    \begin{itemize}
     \item If $d(x,y) > S$, then $k(x,y) = 0$.
     \item If $d(x,y) \leq R$, then $|k(x,y) - 1| < \varepsilon$.
    \end{itemize}
  \end{defi}
  This definition coincides with Yu's original definition in \cite[Definition 2.1]{MR1728880} when $(X,d)$ is a discrete uniformly locally finite metric space. It is well-known that for discrete metric spaces Yu's property $A$ is invariant under coarse equivalence \cite[Proposition~4.2]{MR1871980} \cite[Proposition~1.1.3]{MR2562146}. 

Recall that every locally compact second countable group $G$ admits a proper left-invariant metric $d$ that generates the topology on $G$ and that such a metric is unique up to coarse equivalence (see \cite{MR0348037} and \cite[Theorem 2.8]{HP}). We will say that a metric that generates the topology on $G$ is a \emph{compatible} metric. Moreover, the proper metric space $(G,d)$ has bounded geometry \cite[Lemma 3.3]{HP}. There is a nice connection between Yu's property $A$ and property $A$ on locally compact second countable groups through metric lattices.

\begin{prop}[\cite{MR2115671}, Lemma 2.2]\label{metriclattice-propertyA}
Let $G$ be a locally compact second countable group equipped with a proper left-invariant compatible metric $d$. Then the following are equivalent:
\begin{itemize}
\item[(1)] The metric space $(G,d)$ has property $A$.
\item[(2)] Every uniformly locally finite metric lattice $(Z,d)$ in $G$ has Yu's property $A$.
\item[(3)] There exists a uniformly locally finite metric lattice $(Z,d)$ in $G$ satisfying Yu's property $A$.
\end{itemize}
\end{prop}
When we say that a locally compact second countable group $G$ has property $A$, it will mean that there exists a proper left-invariant compatible metric $d$ such that the metric space $(G,d)$ has property $A$. The proposition above implies that property $A$ on locally compact second countable groups is a coarsely invariant property and it is independent of the choices of the proper left-invariant compatible metric. We refer to \cite{MR2115671, MR3420532} for more on property $A$ for locally compact second countable groups.

There is a well-known operator algebraic characterization of Yu's property $A$ on uniformly locally finite metric spaces in terms of uniform Roe algebras which we now recall. Every $a\in B(\ell^2(Z))$ can be represented as a $Z\times Z$ matrix: $a=[a_{x,y}]_{x,y\in Z}$, where $a_{x,y}:=\ip{a\delta_y}{\delta_x}\in \C$.

We define the \emph{propagation} of $a=[a_{x,y}]_{x,y\in Z}\in B(\ell^2(Z))$ by
\begin{align*}
\text{Prop}(a)=\sup\{d(x,y):x,y\in Z, a_{x,y}\neq 0\}.
\end{align*}
Let $E_R$ be the set of all bounded operators on $\ell^2(Z)$ whose propagations are at most $R$. In fact, $E_R$ is an operator system, i.e.~a self-adjoint closed subspace of $B(\ell^2(Z))$ which contains the unit of $B(\ell^2(Z))$. Moreover, the union $\bigcup_{R>0} E_R$ is a $*$-subalgebra of $B(\ell^2(Z))$ .
\begin{defi}
The $C^*$-algebra defined by the operator norm closure in $B(\ell^2(Z))$
\begin{align*}
C^*_u(Z)=\overline{\bigcup_{R>0}E_R}
\end{align*}
is called the \emph{uniform Roe algebra} of $Z$.
\end{defi}
\begin{thm}[{{\cite[Theorem~5.3]{MR1905840}\label{A_Nuclear} \cite[Theorem 5.5.7]{MR2391387}}}]
Let $(Z,d)$ be a uniformly locally finite metric space. Then the following conditions are equivalent:
\begin{itemize}
\item[(1)] The metric space $(Z,d)$ has Yu's property $A$.
\item[(2)] The uniform Roe algebra $C_u^*(Z)$ is nuclear.
\end{itemize}
\end{thm}
The following definition is due to Yu, see \cite[Definition 11.42]{MR2007488}.
\begin{defi}
An operator $a\in C_u^*(Z)$ is called a \emph{ghost} if $a_{x,y}\ra 0$ as $x,y \ra \infty$. We denote by $G^*(Z)$ the collection of all ghost operators, which forms a closed two sided ideal in $C^*_u(Z)$ and contains the compact operators on $\ell^2(Z)$.
\end{defi}
A natural question is that whether all ghost operators are compact? One can prove that for a uniformly locally finite space with  Yu's property $A$, all ghost operators are compact (\cite[Proposition~11.43]{MR2007488}). Recently, the converse implication was proved by Roe and Willett.
\begin{thm}[{{\cite[Theorem 1.3]{MR3146831}}}]\label{ghost-A}
A uniformly locally finite metric space without Yu's property $A$ always admits non-compact ghosts.
\end{thm}

\section{A Schur multiplier characterization of property $A$} \label{sec:weak amenability and ghost operators}
The purpose of this section is to give a completely bounded Schur multiplier characterization of locally compact groups with property $A$. Consequently, all locally compact weakly amenable groups have property $A$.

Let us start by recalling some definitions. A kernel $k:X\times X\rightarrow \C$ on a nonempty set $X$  is called a $\emph{Schur multiplier}$ if for every operator $a=[a_{x,y}]_{x,y\in X}\in B(\ell^2(X))$ the matrix  $[k(x,y)a_{x,y}]_{x,y\in X}$ represents an operator in $B(\ell^2(X))$, denoted $m_k(a)$. If $k$ is a Schur multiplier, it follows from the closed graph theorem that $m_k$ defines a bounded operator on $B(\ell^2(X))$. We define the Schur norm $\norm{k}_S$ to be the operator norm $\norm{m_k}$ of $m_k$. For instance, any normalized positive type kernel  is a Schur multiplier of norm 1. Here by normalized we mean that $k(x,x) = 1$ for all $x\in X$. The following characterization of Schur multipliers is well-known and is essentially due to Grothendieck.
\begin{thm}[{{\cite[Theorem 5.1]{MR1818047}}}]\label{Schur multipliernorm}
Let $k:X\times X\ra \C$ be a kernel, and let $C\geq 0$. The following conditions are equivalent:
\begin{itemize}
\item[(1)] $k$ is a Schur multiplier with $\norm{k}_S\leq C$.
\item[(2)] There exist a Hilbert space $H$ and two bounded maps $\xi,\eta:X\ra H$ such that $k(x,y)=\ip{\eta_y}{\xi_x}$ for all $x,y \in X$ and $\sup_{x\in X}\norm{\xi_x}\cdot\sup_{y\in X} \norm{\eta_y}\leq C$.
\end{itemize}
\end{thm}
Let $G$ be a locally compact group. A continuous function $\phi:G\ra \C$ is a \emph{Herz--Schur multiplier} if and only if the kernel $\hat{\phi}:G\times G\ra \C$ defined by
\begin{align*}
\hat{\phi}(s,t)=\phi(s^{-1}t),\quad s,t\in G
\end{align*}
is a Schur multiplier on $G$. We denote by $B_2(G)$ the Banach space of Herz--Schur multipliers on $G$ equipped with the Herz--Schur norm $\norm{\phi}_{B_2}=\norm{\hat{\phi}}_S$.

\begin{defi}[Weak amenability \cite{MR996553}]
A locally compact group $G$ is \emph{weakly amenable} if there exists a net $(\phi_i)_{i\in I}$ of continuous, compactly supported Herz--Schur multipliers on $G$, converging uniformly to $1$ on compact sets, and such that $\sup_i \|\phi_i\|_{B_2} < \infty$.
\end{defi}



It is well-known that \emph{discrete} countable weakly amenable groups have Yu's property $A$ \cite{MR1220905, MR1138840} and see also \cite[Remark 3.3]{Anantharaman-Delaroche2009}. In what follows, we will show that the same statement is also true for all locally compact second countable groups. The idea of the proof is to give a completely bounded Schur multiplier characterization for locally compact groups with property $A$.

The next result is an analogue of   \cite[Proposition~11.43]{MR2007488} where we use Schur multipliers in place of positive type kernels in the definition of property $A$. 
\begin{prop}\label{wa-ghost-compact}
Let $(Z,d)$ be a uniformly locally finite metric space. Assume that there exists a constant $C>0$ such that for any $R>0$ and $\epsilon>0$, there exist $S>0$ and a Schur multiplier $k:Z\times Z\ra \C$ with $\norm{k}_S\leq C$ such that
\begin{itemize}
\item If $d(x,y)>S$, then $k(x,y)=0$.
\item If $d(x,y)\leq R$, then $|k(x,y)-1|<\epsilon$.
\end{itemize}
Then every ghost operator in $C_u^*(Z)$ is compact. In particular, it follows from Theorem \ref{ghost-A} that the metric space $(Z,d)$ has Yu's property $A$.
\end{prop}
\begin{proof}
For each $n\in \N$ there exist $S_n>0$ and a Schur multiplier $k_n:Z\times Z\ra \C$ with $\norm{k_n}_S\leq C$ such that $|k_n(x,y)-1|<\frac{1}{n}$ for $d(x,y)\leq n$ and $k_n(x,y)=0$ for $d(x,y)>S_n$. From Theorem \ref{Schur multipliernorm} we note that $\sup_{x,y\in Z}|k_n(x,y)|\leq C$ for all $n\in \N$.

Let $m_{k_n}:B(\ell^2(Z))\ra B(\ell^2(Z))$ be the bounded operator associated with $k_n$. For  every $R>0$ and any $a\in E_R$, it follows that
\begin{align*}
\norm{m_{k_n}(a)-a}\leq \sup_{z\in Z}|B(z,R)|\cdot \norm{a}\cdot \sup_{\{(x,y)\in Z\times Z:\ d(x,y)\leq R\}} |k_n(x,y)-1| \ra 0,\quad \text{as $n\ra \infty$.}
\end{align*}
Since $\norm{m_{k_n}}\leq C$ for all $n\in \N$, we have in fact that $\norm{m_{k_n}(a)-a}\ra 0$ for all $a\in C_u^*(Z)$.

If $H\in C_u^*(Z)$ is a ghost operator, then each $m_{k_n}(H)$ is a compact operator (see \cite[Theorem~3.1]{MR2148492}). So $H$ is compact, as $m_{k_n}(H)\ra H$ in the operator norm.
\end{proof}
The following theorem extends \cite[Theorem~6.1]{MR2945214} to the case of locally compact groups. Given a subset $L$ of a group $G$ we define 
$\Tube (L)$ to be the set 
\[
\Tube(L):= \{(s,t) \in G \times G : s^{-1}t \in L\}.
\]
\begin{thm}\label{wa-a} 
Let $G$ be a locally compact second countable group. The following conditions are equivalent:
\begin{itemize}
\item[(1)] The group $G$ has property $A$.
\item[(2)] If there is a constant $C>0$ such that for any compact subset $K\subseteq G$ and $\epsilon>0$, there exist a compact subset $L\subseteq G$ and a (continuous) Schur multiplier $k:G\times G\rightarrow \C$ with $\norm{k}_S\leq C$ such that $\supp k\subseteq \Tube(L)$ and 
\[
\sup\{|k(s,t)-1|: (s,t)\in \Tube(K)\} <\epsilon.
\]
\end{itemize}
If $G$ is weakly amenable, then in particular the group $G$ has property $A$.
\end{thm}
\begin{proof}
$(1)\Rightarrow (2)$: It follows from \cite[Theorem~2.3]{MR3420532}. In fact, we can take $C=1$ if we assume that the positive type kernel in \cite[Theorem~2.3]{MR3420532} is normalized.

$(2)\Rightarrow (1)$:  Let $d$ be a proper left-invariant compatible metric on $G$ such that the metric space $(G,d)$ has bounded geometry. By the assumption, any uniformly locally finite metric lattice $Z$ in $(G,d)$ satisfies the conditions in Proposition \ref{wa-ghost-compact}, and hence $(Z,d)$ has Yu's property $A$. We complete the proof by applying Proposition \ref{metriclattice-propertyA}.
\end{proof}


\section{Topologically amenable actions and crossed products of $C^*$-algebras} \label{sec:topologically amenable actions}
In this section we recall some basic definitions and state a few results on topologically amenable actions and crossed products of $C^*$-algebras. We should mention that all results presented in this section are well-known to experts and we refer to \cite{MR1799683, MR1926869} for topologically amenable actions and refer to \cite[Section 2]{MR1721796} \cite[Section 7.6]{MR548006} \cite[Section 7.2]{MR2288954} for crossed products of $C^*$-algebras.

Let $\text{Prob}(G)$ denote the space of complex Radon probability measures on a locally compact group $G$. It is the state space of the $C^*$-algebra $C_0(G)$ and it carries two natural topologies: the norm topology and the weak-$*$ topology.
Recall that a locally compact group $G$ acts on a locally compact Hausdorff space $X$ if there exists a homomorphism $\alpha:G\ra \text{Homeo}(X)$ such that the map $G\times X\ra X$ given by $(g,x)\mapsto \alpha(g)(x)$ is continuous. Note that a locally compact group $G$ acts on its probability space $\text{Prob}(G)$ by
\[
(s.m)(U)= m(s^{-1}(U)), 
\]
for all $s\in G$ and any Borel subset $U$ in $G$. 
\begin{defi}[{{\cite{MR1799683}, \cite[Definition 2.1]{MR1926869}}}]
We say that the action $G\curvearrowright X$ is \emph{topologically amenable} if there exists a net $(m_i)_{i\in I}$ of weak-$*$ continuous maps $x\mapsto m_i^x$ from $X$ into the space $\text{Prob}(G)$ such that
\begin{align*}
\lim_i\norm{s.m_i^x-m_i^{s.x}}=0
\end{align*}
uniformly on compact subsets of $X\times G$.
\end{defi}
\begin{defi}[{{\cite[Definition 3.1]{MR1926869}}}]
We say that a locally compact group $G$ is \emph{amenable at infinity} if it admits a topologically amenable action on a compact Hausdorff space $X$.
\end{defi}
We omit the proof of the following lemma as it follows directly from the above definition.
\begin{lem}
Let $X$ and $Y$ be compact Hausdorff $G$-spaces. Assume that there exists a continuous $G$-equivariant map $f:X\ra Y$. If the action $G\curvearrowright Y$ is topologically amenable,  then so is the action $G\curvearrowright X$.
\end{lem}
If a locally compact group $G$ is amenable at infinity, then there are some canonical choices of compact spaces equipped with a topologically amenable action of the group $G$. For instance, let us denote by $C_b^{lu}(G)$ the $C^*$-algebra of bounded left-uniformly continuous functions on $G$. Let $\beta^{lu} (G)$ be the spectrum of $C_b^{lu}(G)$ and it is the universal compact Hausdorff left $G$-space equipped with a continuous $G$-equivariant inclusion of $G$ as an open dense subspace. Let $\partial G:=\beta^{lu} (G) \backslash G$ denote the boundary of the group $G$. It is also a compact Hausdorff space and the left translation action of $G$ on $\beta^{lu} (G)$ restricts to an action on $\partial G$. The inclusion map from $\partial G$ into $\beta^{lu} (G)$ is clearly equivariant. We obtain the following result from the lemma stated above.
\begin{prop}[{{\cite[Proposition~3.4]{MR1926869}}}]\label{amenat_infinity}
Let $G$ be a locally compact group. The following conditions are equivalent:
\begin{itemize}
\item[(1)] $G$ is amenable at infinity.
\item[(2)] The left translation action of $G$ on $\beta^{lu} (G)$ is topologically amenable.
\item[(3)] The left translation action of $G$ on $\partial G$ is topologically amenable.
\end{itemize}
\end{prop}
A $G$-$C^*$-algebra consists of a $C^*$-algebra $A$, a locally compact group $G$, and a group homomorphism $\alpha \colon G \to \mathrm{Aut}(A)$ such that the map $g \mapsto \alpha_g(a)$ is continuous for all $a \in A$.

\begin{defi}
Let $G$ be a locally compact group and $A$ be a $G$-$C^*$-algebra equipped with the action $\alpha$. A \emph{covariant representation} of the $G$-$C^*$-algebra $A$ is a pair $(\pi,U)$ where $\pi:A\ra B(H)$ is a $*$-homomorphism and $U:G\ra B(H)$ is a unitary representation of $G$ such that $U_s\pi(a)U_s^*=\pi(\alpha_s(a))$ for all $s\in G$ and $a\in A$.
\end{defi}
We denote by $C_c(G,A)$ the vector space of continuous $A$-valued functions on $G$ with compact support. Define a convolution product and involution on $C_c(G,A)$ by 
\begin{align*}
f*g(s)=\int_G f(t)\alpha_t(g(t^{-1}s))\ d\mu(t) & &\text{and}& & f^*(s)=\frac{\alpha_s(f(s^{-1})^*)}{\Delta(s)},
\end{align*}
where $\Delta$ is the modular function and $\mu$ is a fixed left Haar measure on $G$, respectively. In this way, $C_c(G,A)$ becomes a $*$-algebra.

Given a covariant representation $(\pi,U)$ of a $G$-$C^*$-algebra $A$ on a Hilbert space $H$. Then
\begin{align*}
\pi\rtimes U(f)=\int_G\pi(f(s))U_s\ d\mu(s)
\end{align*}
defines a $*$-representation of $C_c(G,A)$ on the Hilbert space $H$. If $\pi \rtimes U$ is faithful then we define a norm on $C_c(G,A)$ by $\norm{f}_{(\pi,U)}: = \norm{\pi \rtimes U (f)}$. The completion of $C_c(G,A)$ with respect to this norm is denoted by $A \rtimes_{(\pi,U)} G$.
\begin{defi}
Let $G$ be a locally compact group and $A$ be a $G$-$C^*$-algebra equipped with the action $\alpha$ of $G$. The \emph{full crossed product} $A\rtimes_\alpha G$ is the completion of $C_c(G,A)$ with respect to the universal $C^*$-norm $\norm{\cdot}_u$ given by
\begin{align*}
\norm{f}_u:=\sup\{ \norm{\pi\rtimes U(f)}: \text{$(\pi,U)$ is a covariant represention of $A$}\}.
\end{align*}
\end{defi}
It follows from the definition of the full crossed product that $\pi\rtimes U$ extends to a  $*$-representation of $A\rtimes_\alpha G$ for every covariant representation $(\pi,U)$ of a $G$-$C^*$-algebra $A$.

To define the reduced crossed product, we begin with a \emph{faithful} $*$-representation $\pi$ of the $G$-$C^*$-algebra $A$ on a Hilbert space $H$. Define a $*$-representation $\pi_\alpha^A$ of $A$ on $L^2(G,H)$ by
\begin{align*}
(\pi_\alpha^A(a)\xi)(t)=\pi(\alpha_{t^{-1}}(a))\xi(t),
\end{align*}
for $a\in A$, $t\in G$ and $\xi\in L^2(G,H)$. Let $\lambda$ denote the left regular representation of $G$ on $L^2(G)$. Then $(\pi_\alpha^A, \lambda\otimes 1)$ is a covariant representation of $A$. The \emph{regular representation} $\pi_\alpha^A\rtimes (\lambda\otimes 1)$ of $C_c(G,A)$ on $L^2(G,H)$ is faithful. In particular, the universal $C^*$-norm $\norm{\cdot}_u$ is an honest norm.
\begin{defi}
Let $G$ be a locally compact group and $A$ be a $G$-$C^*$-algebra equipped with the action $\alpha$ of $G$. The \emph{reduced crossed product} $A\rtimes_{\alpha,r}G$ is the completion of $C_c(G,A)$ with respect to the reduced $C^*$-norm $\norm{\cdot}_r$ given by
\begin{align*}
\norm{f}_r:=\norm{\pi_\alpha^A\rtimes (\lambda\otimes 1)(f)}.
\end{align*}
\end{defi}
It is well-known that the reduced crossed product $A\rtimes_{\alpha,r}G$ does not depend on the choice of the faithful representation $\pi:A\ra B(H)$ (see \cite[Section 2]{MR1721796}). Moreover, we have a natural surjective $*$-homomorphism $A\rtimes_{\alpha}G\ra A\rtimes_{\alpha,r}G$.

For a given locally compact group $G$, the full crossed product $(-)\rtimes G$ and the reduced crossed product $(-)\rtimes_r G$ form functors from the category of $G$-$C^*$-algebras to the category of $C^*$-algebras. Indeed, let $A$ and $B$ be $G$-$C^*$-algebras with actions $\alpha$ and $\beta$, respectively. For every equivariant $*$-homomorphism $\theta:A\ra B$ there are two canonical $*$-homomorphisms 
\[
\theta_u: A\rtimes_\alpha G\ra B\rtimes_\beta G \quad \text{and}\quad \theta_r:A\rtimes_{\alpha,r} G\ra B\rtimes_{\beta,r} G.
\]
The two $*$-homomorphisms are induced by the map on the dense $*$-subalgebras $C_c(G,A) \to C_c(G,B)$ defined by $f \mapsto \theta \circ f$.
\begin{exem}
Let $C_b(G)$ be the space of bounded continuous complex valued functions on a locally compact group $G$. Let $M: C_b(G)\ra B(L^2(G))$ be the multiplication operator on $L^2(G)$ given by
\begin{align*}
(M(f)\xi)(x)=f(x)\xi(x),
\end{align*}
where $f\in C_b(G)$, $\xi\in L^2(G)$ and $x\in G$. It is clear that $M$ is a faithful $*$-representation. Let $L$ and $R$ be the left and right translations on $C_b(G)$, respectively. More precisely,
\[
(L_gf)(x)=f(g^{-1}x)\quad \text{and} \quad (R_gf)(x)=f(xg),
\]
for $f\in C_b(G)$ and $x,g\in G$.
We denote the space of bounded left (right) uniformly continuous functions on $G$ by $C_b^{lu}(G)$ (respectively $C_b^{ru}(G)$). We have the left and right regular representations $\lambda, \rho \colon G \to U(L^2(G))$ given by
\[
(\lambda_g\xi)(x) = \xi(g^{-1}x)\quad  \text{and} \quad (\rho_g \xi)(x) = \xi(xg) \Delta(g)^{1/2},
\]
where $\xi \in L^2(G)$ and $g,x \in G$. It follows that $(M, \lambda)$ and $(M, \rho)$ are covariant representations of the $C^*$-dynamical systems $(C_b^{lu}(G), G, L)$ and $(C_b^{ru}(G), G, R)$, respectively.
\end{exem}
For $f \in C_c(G,C_b(G))$, $f_g$ denotes the continuous bounded function of $f$ at the point $g \in G$ and $f_g(h)$ the value of $f_g$ at the point $h \in G$. Strictly speaking $C_b^{lu}(G) \rtimes_{L,r} G$ and $C_b^{lu}(G) \rtimes_{(M, \lambda)} G$ are represented on two different Hilbert spaces. However, we include a proof of the following standard result. 
\begin{prop}
In the notation introduced above, we have the following $*$-isomorphisms: 
\[
C_b^{ru}(G) \rtimes_{(M, \rho)} G \cong C_b^{lu}(G) \rtimes_{(M, \lambda)} G \cong C_b^{lu}(G) \rtimes_{L,r} G \cong C^{ru}_b(G) \rtimes_{R,r} G.
\]
\end{prop}
\begin{proof}
 Let $\psi$ be the $G$-equivarant $*$-isomorphism $\psi \colon C_b^{ru}(G) \to C_b^{lu}(G)$, $\psi(f)(g) = f(g^{-1})$. Let $\bar{\psi}$ be the extension of $\psi$ to 
 \[
 \bar{\psi}: C_c(G,C^{ru}_b(G)) \to C_c(G,C^{lu}_b(G))
 \] 
 where $\bar{\psi}(f)_s(g) = f_s(g^{-1})$. Let $U \colon L^2(G) \to L^2(G)$, $(U \xi)(g) = \xi(g^{-1}) \Delta(g^{-1})^{1/2}$ for all $\xi \in L^2(G)$ and $x \in G$. Then the isomorphism 
 \[
 C_b^{ru}(G) \rtimes_{(M, \rho)} G \cong C_b^{lu}(G) \rtimes_{(M, \lambda)} G
 \]
 is given by the map $f \mapsto U \psi(f) U^*$ for all $f \in C_c(G,C^{ru}_b(G))$.

Now fix a unit vector $\eta_0 \in L^2(G)$ and define a map
\[
V_1 \colon L^2(G) \to L^2(G,L^2(G))
\]
by 
\[
(V_1 \xi)_g(x) = \xi(g) \eta_0(x)
\]
 for all $\xi \in L^2(G)$. Similarly, we define
 \[
 V_2\colon L^2(G,L^2(G)) \to L^2(G,L^2(G)); \quad (V_2 \xi)_g(x) = \xi_{gx^{-1}} \Delta(x^{-1})^{1/2}. 
 \]
 Finally, let $V = V_2^* V_1$. Then there is an isomorphism 
 \[
 C_b^{lu}(G) \rtimes_{(M, \lambda)} G \cong C_b^{lu}(G) \rtimes_{L,r} G
 \]
 given  by $f \mapsto V f V^*$ for all $f \in C_b^{lu}(G) \rtimes_{(M, \lambda)} G$. Let $\theta = \psi^{-1}$. Then $\theta$ is an isomorphism and extends to a $*$-isomorphism 
\[
\theta_r \colon C^{lu}_b(G) \rtimes_{L,r} G \to C^{ru}_b(G) \rtimes_{R,r} G. \qedhere
\]
\end{proof}

From this point on we shall use the isomorphisms $C_b^{lu}(G) \rtimes_{(M, \lambda)} G \cong C_b^{lu}(G) \rtimes_{L,r} G$ and $C_b^{ru}(G) \rtimes_{(M, \rho)} G \cong C_b^{ru}(G) \rtimes_{R,r} G$ without further mention. Moreover, $M\rtimes_r \rho$ induces a $*$-isomorphism between $C_0(G)\rtimes_{R,r} G$ and the $C^*$-algebra of compact operators $K(L^2(G))$  \cite[Theorem 4.24]{MR2288954}. We can also conclude the same facts for $(M, \lambda)$ on $(C_b^{lu}(G), G, L)$.
%
%


We end this section with an important theorem, which will be used in the next section. 
\begin{thm}[{{\cite{MR919508} \cite[Theorem 5.3 and Theorem 5.8]{MR1926869}\label{ADtop-dyn}}}]
Let $G$ be a locally compact group and $X$ be a locally compact Hausdorff $G$-space. Consider the following conditions:
\begin{itemize}
\item[(1)] The action of $G$ on $X$ is topologically amenable.
\item[(2)] $(C_0(X)\otimes A)\rtimes G=(C_0(X)\otimes A)\rtimes_r G$ for every $G$-$C^*$-algebra $A$.
\item[(3)] $C_0(X)\rtimes_r G$ is nuclear.
\end{itemize}
Then $(1)\Rightarrow (2)\Rightarrow (3)$. Moreover, $(3)\Rightarrow (1)$ if $G$ is discrete.
\end{thm}
\begin{rem}
In the theorem above, the condition (3) does not imply the condition (1) in general. For example, the reduced group $C^*$-algebra $C_r^*(G)$ is always nuclear if the group $G$ is connected \cite[Corollary 6.9(c)]{MR0454659} or Type I \cite[Section 2]{MR974319}.
\end{rem}
\section{Exactness of locally compact groups} \label{sec:exactness of locally compact groups}
In this section we show that exactness of a locally compact second countable group is equivalent to amenability at infinity. This solves an open problem raised by Anantharaman-Delaroche in \cite[Problem 9.3]{MR1926869}. 

We would like to point out that a locally compact group that is amenable at infinity is known to be exact \cite[Theorem 7.2]{MR1926869}. The converse implication is well-established for discrete groups \cite{MR1721796,MR1763912,MR1926869}. The situation in the locally compact case is quite different from the discrete case as the existing proofs do not translate directly to the locally compact case. Let us begin this section by recalling the definition of exact groups:
\begin{defi}[{{\cite{MR1721796}}}]
We say that a locally compact group $G$ is \emph{exact}, if the reduced crossed product functor $A\ra A\rtimes_{r}G$ is exact for any $G$-$C^*$-algebra $A$. To be more precise, for every $G$-equivariant short exact sequence of $G$-$C^*$-algebras $0\ra I\ra A\ra A/I\ra 0$, the corresponding sequence $$0\ra I\rtimes_{r}G\ra A\rtimes_{r}G\ra A/I\rtimes_{r}G\ra 0$$ of reduced crossed products is still exact.
\end{defi}
\begin{rem}
The corresponding morphisms $\iota_r: I\rtimes_{r}G\ra A\rtimes_{r}G$ and $q_r: A\rtimes_{r}G\ra A/I\rtimes_{r}G$ are always injective and surjective, respectively. Moreover, $\Im\ \iota_r \subseteq \ker q_r$. So the group $G$ is exact if and only if this inclusion is  an equality for every $G$-$C^*$-algebra $A$ and every $G$-invariant closed ideal $I$ in $A$. However, the full crossed product functor is \emph{always} exact by its universal property.
\end{rem}
Exact groups include all locally compact amenable groups \cite[Proposition 6.1]{MR1725812}, almost connected groups \cite[Corollary 6.9]{MR1725812}, countable linear groups \cite[Theorem 6]{MR2217050}, word hyperbolic groups \cite[Theorem 5.1]{MR1293309}, and so on. Nonetheless, there exist finitely generated non-exact groups \cite{MR1978492, AD, Osa}.

We now identify all elements in $A\rtimes_r G$ which are in $\ker q_r$. The next proposition provides a useful criterion for this in terms of slice maps. Recall that for any normal linear functional $\psi\in B(L^2(G))_*$, the slice map $S_\psi$ corresponding to $\psi$ is the bounded map defined as follows
\begin{align*}
S_\psi: A\rtimes_{\alpha,r} G\xrightarrow{\pi^A_\alpha\rtimes(\lambda\otimes 1)} B(L^2(G,H))\cong B(H)\bar{\otimes} B(L^2(G))\xrightarrow{\mathrm{id}_{B(H)}\bar{\otimes}\psi} B(H),
\end{align*}
where $\pi_\alpha^A\rtimes (\lambda\otimes 1)$ is the regular representation associated to the reduced crossed product $A\rtimes_{\alpha,r} G$. If $\psi=\omega_{\xi,\eta}$, where $\xi,\eta \in C_c(G)$ and $\omega_{\xi,\eta}(x)=\ip{x\xi}{\eta}$ for $x\in B(L^2(G))$, then it follows from the proof of \cite[Lemma 2.1]{MR1721796} that
\begin{align}\label{sliceformula}
S_\psi(f)=\int_G\int_G\xi(g^{-1}h)\bar{\eta(h)}\alpha_{h^{-1}}(f(g)) \ d\mu(g)d\mu(h),\quad f\in C_c(G,A).
\end{align}

\begin{prop}[{{\cite[Proposition~2.2]{MR1721796}}}]\label{slicemap}
Let $0 \to I \to A \to A/I \to 0$ be a $G$-equivariant exact sequence of $G$-$C^*$-algebras. For $x\in A\rtimes_r G$, the following are equivalent:
\begin{itemize}
\item[(1)] $x\in \ker q_r$.
\item[(2)] $S_\psi(x)\in I$ for all $\psi\in B(L^2(G))_*$.
\item[(3)] $S_{\omega_{\xi,\eta}}(x)\in I$ for all $\xi,\eta \in C_c(G)$.
\end{itemize}
\end{prop}
In order to prove our main theorem of this paper, we need to establish a connection between $C_b^{ru}(G)\rtimes_{R,r} G$ of a locally compact group $G$ and the uniform Roe algebra $C_u^*(Z)$ of a metric lattice $Z$ in the ambient group $G$.

Since $Z$ is uniformly discrete, we fix $\delta > 0$ such that for all $z,w \in Z$, $d(z, w) \geq \delta$, whenever $z \neq w$. Let $\varphi$ be a continuous compactly supported positive valued function on $G$ such that $\supp \varphi \subseteq B(e, \delta/4)$ and $\norm{\phi}_2=1$. For $z \in Z$, set $\varphi_z$ to be the function $g \mapsto \varphi(z^{-1}g)$ for $g \in G$. Clearly, each $\varphi_z$ is supported on a $\delta / 4$-neighbourhood around $z$. As $Z$ is $\delta$-uniformly discrete, the functions in the set $\{\varphi_z\}_{z \in Z}$ have pairwise disjoint support. In particular, $\{\varphi_z : z \in Z\}$ forms an orthonormal set in $L^2(G)$.
 
Define a linear isometry $W \colon \ell^2(Z) \to L^2(G)$ by the formula $\delta_z \mapsto \varphi_z$, which sends an orthonormal basis to an orthonormal set. For $\eta \in \ell^2(Z)$, we have that
\[
(W\eta)(x) = \sum_{z \in Z} \eta(z) \varphi_z(x)
\]
 for all $x \in G$. Moreover,
 \[
 (W^* \xi)(z) = \int_G \xi(y) \varphi_z(y) \, d\mu(y),\  \text{for $\xi \in L^2(G)$ and $z\in Z$.}
 \] 
Let $T \in C^*_u(Z)$ be a finite propagation operator and denote $\ip{T \delta_w}{\delta_z}$ by $T_{z,w}$. By the left invariance of the Haar integral we have that for every $x \in G$ and all $\xi \in L^2(G)$,
\begin{align*}
  (WTW^*)(\xi)(x) &= \sum_{z\in Z}\phi_z(x)\sum_{w\in Z}T_{z,w}\int_G \xi(y)\phi_w(y)\ d\mu(y)\\
                  &=\int_G \sum_{z,w \in Z} \varphi_z(x) \varphi_w(xy) T_{z,w} \xi(xy) \, d\mu(y),
\end{align*}
where we can switch the order of summation and integration as the sums are finite because $T$ has finite propagation and $Z$ is uniformly locally finite. 

Now, we define a function $\hat{T}:G\times G\ra \C$ given by
\begin{align}\label{hat(T)}
 \widehat T_y(x) = \sum_{z,w \in Z} \varphi_z(x) \varphi_w(xy) T_{z,w} \Delta(y)^{-1/2} \ \text{for $(y,x) \in G\times G$.}
\end{align}
In order to ease notation we will define 
\[
\sigma_{x,y}(z,w) :=   \varphi_z(x) \varphi_w(xy) T_{z,w} \Delta(y)^{-1/2} 
\]
for all $x,y \in G$ and $z,w \in Z$. 

\begin{rem} \label{rem:T hat is reduced to a single term}
Observe that if $\sigma_{x,y}(z,w)$ is non-zero then $d(x,z) \leq \delta / 4$ and $d(xy,w) \leq \delta / 4$. The metric lattice $Z$ is $\delta$-uniformly discrete so if $\sigma_{x,y}(z,w)$ and $\sigma_{x,y}(z',w')$ are both non-zero for some $z,z',w,w' \in Z$ then necessarily $(z,w)= (z',w')$. Hence if $x,y \in G$ such that $\hat{T}_y(x) \neq 0$ then there exists a unique pair $(z,w) \in Z \times Z$ such that $\hat{T}_y(x) = \sigma_{x,y}(z,w)$. Otherwise, if $\hat{T}_y(x) = 0$ then $\hat{T}_y(x) = \sigma_{x,y}(z,w) = 0$ for all $z,w \in Z$.

If $d(z,w) > \mathrm{Prop}(T)$ then $\sigma_{x,y}(z,w) = 0$. By the triangle inequality
\[
 d(e,y) = d(x,xy) \leq d(z,w) + d(x,z) + d(w,xy).
\]
Hence if $\sigma_{x,y}(z,w) \neq 0$ then $ d(z,w) \leq \mathrm{Prop}(T)$ and $d(x,z) + d(w,xy) \leq \delta/ 2$. Thus if $d(e,y) > \mathrm{Prop}(T) + \delta /2$ then $\sigma_{x,y}(z,w) = 0$ for all $z,w \in Z$ and $x \in G$.

Further observe that if $x,x',y,y' \in G$ such that $\hat{T}_y(x) \neq 0$, $\hat{T}_{y'}(x') \neq 0$ and $d(x,x') + d(xy,x'y') \leq \delta / 4$ then there exists only one pair $(z,w) \in Z \times Z$ such that $\hat{T}_{y}(x) = \sigma_{x,y}(z,w)$ and $\hat{T}_{y'}(x') = \sigma_{x',y'}(z,w)$. This follows from the triangle inequality and that $Z$ is $\delta$-uniformly discrete. 
\end{rem}




\begin{lem}
 Let $G$ be a locally compact second countable group and $Z$ a metric lattice. If $T \in C^*_u(Z)$ has finite propagation then $\hat{T} \in C_c(G,C^{ru}_{b}(G))$, where $\hat{T}$ is defined in \eqref{hat(T)}.
\end{lem}

\begin{proof}
In this proof we will show that the function $y \mapsto \hat{T}_y$ is compactly supported and continuous and $\hat{T}_y$ is bounded and right uniformly continuous for all $y \in G$.

We first show the map $y \mapsto \hat{T}_y$ is compactly supported. If $G$ is compact then clearly the map is compactly supported. Assume $G$ is not compact and let $y \in G$ be such that $d(e,y) > \mathrm{Prop}(T) + \delta/2$. Suppose for a contradiction that $\supp{\hat{T}_y} \neq \emptyset$. Then there exists $x \in G$ such that $\hat{T}_y(x) \neq 0$. By Remark \ref{rem:T hat is reduced to a single term} there exists a unique pair $(z,w) \in Z \times Z$ such that $\hat{T}_y(x) = \sigma_{x,y}(z,w)$. As $d(x,xy) > \mathrm{Prop}(T) + \delta/2$, it follows that $\hat{T}_y(x) = \sigma_{x,y}(z,w) = 0$ which contradicts the choice of $x \in G$. Hence $\supp{\hat{T}_y} = \emptyset$ whenever $d(e,y) > \mathrm{Prop}(T) + \delta/2$.

Now we show that for all $y \in G$, $\hat{T}_y$ is bounded. If $y \notin \supp{\hat{T}}$ then clearly $\hat{T}_y$ is bounded so we assume $y \in \supp{\hat{T}}$. Let $M > 0$ be such that $\sup \{\Delta(y)^{-1/2} : y \in \supp{\hat{T}}\} < M < \infty$. This is well-defined because $\hat{T}$ is compactly supported. Let $x \in G$ be such that $\hat{T}_y(x) \neq 0$. By Remark \ref{rem:T hat is reduced to a single term} it follows that there exists a unique pair $(z,w) \in Z \times Z$ such that $\hat{T}_y(x) = \sigma_{x,y}(z,w)$. Hence for all $x \in G$,
\[
 |\hat{T}_y(x)| \leq \norm{\varphi}^2_{\infty}\norm{T}M.
\]
%
%
We will now show the continuity conditions so fix $\varepsilon > 0$ for the remainder of the proof. First we show that for all $y \in G$, $\hat{T}_y$ is right uniformly continuous. If $\hat{T}_y$ is equal to the zero function then clearly $\hat{T}_y$ is right uniformly continuous so we assume $\hat{T}_y$ is not equal to the zero function. Since $\varphi$ is uniformly continuous on the metric space $(G,d)$, there exists $\delta' > 0$ such that for every $x,x'\in G$ with $d(x,x')<\delta'$, we have that 
\[
\norm{T}|\varphi(x) - \varphi(x')| < \frac{\varepsilon }{2\norm{\varphi}_{\infty} M}. 
\]
 Choose $g \in G$ such that $d(e,g) + d(e,y^{-1} g y)< \min \{\delta', \delta / 4 \}$. For $x \in G$ there exist $(z,w), (z',w') \in Z \times Z$, depending on $x$ and $g$, such that
\begin{equation*}
 |R_g \hat{T}_y(x) - \hat{T}_y(x)| = |\hat{T}_y(xg) - \hat{T}_y(x)|= |\sigma_{xg,y}(z,w) - \sigma_{x,y}(z',w')|.
\end{equation*}
As we have chosen $g \in G$ such that $d(x,xg) + d(y,gy) < \delta / 4$ it follows by Remark \ref{rem:T hat is reduced to a single term} that $(z,w) = (z',w')$. Hence,
 \begin{multline*}
 |R_g \hat{T}_y(x) - \hat{T}_y(x)| \leq |T_{z,w}| |\varphi_z(xg) - \varphi_z(x)| \varphi_w(xgy)  \Delta(y)^{-1/2} \\ + |T_{z,w}||\varphi_w(xgy) - \varphi_w(xy)| \varphi_z(x) \Delta(y)^{-1/2} < \varepsilon.
\end{multline*}
It follows that $|R_g \hat{T}_y(x) - \hat{T}_y(x)| < \varepsilon$ for all $x \in G$ and because $g$ does not depend on $x$ it follows that $\hat{T}_y \in C^{ru}_b(G)$ for all $y \in G$.

The modular function is continuous so for all $y_0 \in G$ there exists $\delta''(y_0)>0$ such that for all $y \in G$ if $d(y,y_0) < \delta''(y_0)$ then 
\[
\norm{T} |\Delta(y)^{-1/2} - \Delta(y_0)^{-1/2}| <\frac{ \varepsilon}{2 \norm{\varphi}_{\infty}^2}.
\]
 We now complete the proof by showing the function $y \mapsto \hat{T}_y$ is continuous. Fix $y_0 \in G$ and choose $y \in G$ such that $d(y, y_0) <  \min \{ \delta' , \delta''(y_0), \delta/4 \}$. For $x \in G$, there exists pairs $(z,w),(z',w') \in Z \times Z$ such that 
\[
|\hat{T}_y(x)-\hat{T}_{y_0}(x)| = |\sigma_{x,y}(z,w) - \sigma_{x,y_0}(z',w')|.
\]
By Remark \ref{rem:T hat is reduced to a single term} it follows that if $x,y, y_0 \in G$ such that $\hat{T}_y(x) \neq 0$ and $\hat{T}_{y_0}(x) \neq 0$  then the pairs $(z,w)$ and $(z',w')$ are forced to equal each other. Otherwise we can choose the pairs $(z,w)$ and $(z',w')$ to equal each other. Hence
%
\begin{multline*}
 |\hat{T}_y(x)-\hat{T}_{y_0}(x)| < |T_{z,w}| |\varphi_w(xy) - \varphi_{w}(xy_0)| \varphi_{z}(x)  \Delta(y)^{-1/2}\\ + |T_{z,w}| |\Delta(y)^{-1/2} - \Delta(y_0)^{-1/2}| \varphi_{w}(xy_0) \varphi_z(x)  < \varepsilon.
\end{multline*}
Thus $|\hat{T}_y(x)-\hat{T}_{y_0}(x)| < \varepsilon$ for all $x \in G$ so the function $y \mapsto \hat{T}_y$ is continuous.
%
%
%
%
\end{proof}

\begin{prop}\label{mainprop}
Let $G$ be a locally compact second countable group equipped with a proper left-invariant compatible metric $d$. If $Z$ is a uniformly locally finite metric lattice in the metric space $(G,d)$, then there exist a faithful $*$-homomorphism $\Phi:C_u^*(Z)\ra C_b^{ru}(G)\rtimes_{R,r} G$ and a c.c.p. map $E: C_b^{ru}(G)\rtimes_{R,r} G\ra C_u^*(Z)$ satisfying the following properties:
\begin{itemize}
\item[(1)] $E\circ \Phi= \emph{Id}_{C_u^*(Z)}$.
\item[(2)] $T\in K(\ell^2(Z))$ if and only if $\Phi(T)\in C_0(G)\rtimes_{R,r} G$.
\end{itemize}
\end{prop}
\begin{proof}

Let $M:C_b^{ru}(G)\ra B(L^2(G))$ be the multiplication operator on $L^2(G)$ and $\rho:G\ra B(L^2(G))$ be the right regular representation. We will show that for a finite propagation operator $T \in C^*_u(Z)$, the operator $WTW^*$ in $B(L^2(G))$ is the image of $\hat{T}$ under the faithful $*$-representation $M\rtimes_r\rho:C_b^{ru}(G)\rtimes_{r,R} G\ra B(L^2(G))$.
Indeed, for all $\xi \in L^2(G)$ and $x \in G$ we have that
\begin{align}\label{111}
 (M \rtimes_r \rho)(\widehat T)(\xi)(x)= \int_G \widehat T_y(x) \xi(xy) \Delta(y)^{1/2} \, d\mu(y) =(WTW^*)(\xi)(x).
\end{align}

Since the image of $M \rtimes_r \rho$ is closed, we conclude that $WC_u^*(Z)W^*$ is contained in the image of $M \rtimes_r \rho$. Hence, there is a well-defined faithful $*$-homomorphism $$\Phi:C_u^*(Z)\ra C_b^{ru}(G)\rtimes_{R,r} G$$ defined by 
 \[
 \Phi(T)=(M \rtimes_r \rho)^{-1}(WTW^*).
 \]
  If $T$ has finite propagation, then it follows from the calculation \ref{111} that $\Phi(T)=\hat{T}$, where $\hat{T}$ is given by the formula \ref{hat(T)}.
  
 We now define a c.c.p. map $E:C_b^{ru}(G)\rtimes_{R,r} G \ra B(\ell^2(Z))$
by 
\[
E(a)=W^*(M \rtimes_r \rho(a))W,
\]
for $a\in C_b^{ru}(G)\rtimes_{R,r} G$. 

We claim that the image of $E$ is contained in $C_u^*(Z)$. Indeed, let $f\in C_c(G,C_b^{ru}(G))$ with support in $B(e,s)$ for some $s>0$, then
\begin{align*}
\ip{W^*(M\rtimes_r \rho (f))W\delta_w}{\delta_z}& =\ip{M\rtimes_r \rho (f)(\phi_w)}{\phi_z}\\
& =\int_{G} \int_{G} f_y(x)\phi_w(xy)\Delta(y)^{1/2}\phi_z(x)\,d\mu(y)d\mu(x).
\end{align*}
If $\ip{W^*(M\rtimes_r \rho (f))W\delta_w}{\delta_z}\neq 0$, then there exist $x,y$ in $G$ such that $f_y(x)\phi_w(xy)\phi_z(x)\neq 0$. It follows that
$$
d(z,w)\leq d(z,x)+d(x,xy)+d(xy,w)\leq \delta/4+d(e,y)+\delta/4\leq s+\delta/2,
$$
where the last inequality follows from the fact that $f_y(x)\neq 0$. In particular, $E(f)$ has propagation at most $s + \delta/2$. So the image of $E$ is contained in $C_u^*(Z)$ as desired.

Property (1) in the statement of the Proposition now follows from the constructions of $\Phi$ and $E$. 
Furthermore, since $W$ is an isometry and $M \rtimes_r \rho$ induces a $*$-isomorphism between $C_0(G)\rtimes_{R,r}G$ and $K(L^2(G))$, property (2) in the statement follows easily from the fact that the compact operators form a two-sided ideal.
\end{proof}

\begin{prop}\label{mainprop2} 
Let $G$ be a locally compact second countable group equipped with a proper left-invariant compatible metric $d$, and let $Z$ be a uniformly locally finite metric lattice in the metric space 
$(G,d)$. Furthermore, let $\Phi:C_u^*(Z)\ra C_b^{ru}(G)\rtimes_{R,r} G$ be the map constructed in Proposition \ref{mainprop}. Then the image of the ghost ideal $G^*(Z)$ under $\Phi$ is contained in the kernel of the map 
\[
q_r: C_b^{ru}(G)\rtimes_{R,r} G\ra (C_b^{ru}(G)/C_0(G))\rtimes_{R,r} G.
\]
\end{prop}
\begin{proof}
Let $H\in G^*(Z)$. By Proposition \ref{slicemap}, it is sufficient to show that $S_{\omega_{\xi,\eta}}(\Phi(H))\in C_0(G)$ for all $\xi,\eta\in C_c(G)$, where $S_{\omega_{\xi,\eta}}$ is the slice map given by the formula \ref{sliceformula}.
 
Given $\epsilon>0$ and $\xi, \eta \in C_c(G)$, we need to find $C>0$ such that
\begin{align*}
d(e,x)>C\Rightarrow |S_{\omega_{\xi,\eta}}(\Phi(H))(x)|<\epsilon.
\end{align*}
Consider the bounded continuous kernel $k:G\times G\ra \C$ given by $k(g,h)=\xi(g^{-1}h)\overline{\eta(h)}$. Since both $\xi$ and $\eta$ are compactly supported, there exist two compact subsets $K_1$ and $K_2$ such that $\supp k\subseteq K_1\times K_2$. Let $D$ be a positive number such that
\begin{align*}
d(g,h)+d(e,h)+\Delta(g)^{-1/2}\leq D,
\end{align*}
for all $(g,h)\in K_1\times K_2$. Choose a small $\epsilon'>0$ such that
\begin{align*}
3\epsilon' D\norm{\xi}_\infty \norm{\eta}_\infty\mu(K_1)\mu(K_2)\norm{\phi}^2_\infty \leq \epsilon.
\end{align*}
Since $H$ is a ghost, we choose a $N>0$ such that if $z,w\notin B(e,N)$ then $|H_{z,w}|<\epsilon'$. As the slice map $S_{\omega_{\xi,\eta}}$ is continuous, we can choose an operator $T\in C_u^*(Z)$ of finite propagation such that 
\[
\norm{T-H}_{B(\ell^2(Z))}+\norm{S_{\omega_{\xi,\eta}}(\Phi(T)-\Phi(H))}_\infty<\text{min}\{\epsilon/3,\epsilon'\}.
\]
 In particular, $|T_{z,w}-H_{z,w}|<\epsilon'$ for all $z,w \in Z$. Recall that $\Phi(T)=\hat{T}$, which is given by the formula \eqref{hat(T)}, and for all $x \in G$,
\begin{align*}
S_{\omega_{\xi,\eta}}(\hat{T})(x)&=\int_G\int_G\xi(g^{-1}h)\overline{\eta(h)}R_{h^{-1}}(\hat{T}_g)(x)\ d\mu(g)d\mu(h)\\
                                 &=\int_G\int_G k(g,h)\sum_{z,w\in Z}\phi_z(xh^{-1})\phi_w(xh^{-1}g)T_{z,w}\Delta(g)^{-1/2}\  d\mu(g)d\mu(h).                          
\end{align*}

Set $C=N+D+\delta/4$. For each $x \notin B(e,C)$ we show that $|S_{\omega_{\xi,\eta}}(\Phi(H))(x)|<\epsilon$. If $S_{\omega_{\xi,\eta}}(\Phi(H))(x)= 0$, we are done. Otherwise, there exist $g\in K_1$ and $h\in K_2$ such that $\sum_{z,w\in Z}\phi_z(xh^{-1})\phi_w(xh^{-1}g)T_{z,w}\Delta(g)^{-1/2}\neq 0$. By Remark \ref{rem:T hat is reduced to a single term} there exists only one unique pair $(z(h),w(g,h)) \in Z \times Z$, depending on $g,h \in G$, such that $\phi_{z(h)}(xh^{-1})$ and $\phi_{w(g,h)}(xh^{-1}g)$ are both non-zero. In particular, $(z(h),w(g,h)) \in B(xh^{-1},\delta/4)\times B(xh^{-1}g,\delta/4)$ and
\[
S_{\omega_{\xi,\eta}}(\hat{T})(x)=\int_{K_2}\int_{K_1} k(g,h)\phi_{z(h)}(xh^{-1})\phi_{w(g,h)}(xh^{-1}g)T_{z(h),w(g,h)}\Delta(g)^{-1/2}\  d\mu(g)d\mu(h).
\] 
We see that $|T_{z(h),w(g,h)}|\leq |T_{z(h),w(g,h)}-H_{z(h),w(g,h)}|+|H_{z(h),w(g,h)}|<2\epsilon'$, because
\begin{align*}
d(e,z(h))&\geq d(x,e)-d(h,e)-d(z(h),xh^{-1})>C-D-\delta/4=N,\\
d(e,w(g,h))&\geq d(e,x)-d(xh^{-1}g,x)-d(w(g,h),xh^{-1}g)>C-D-\delta/4=N.
\end{align*}
It follows that
\begin{align*}
|S_{\omega_{\xi,\eta}}(\hat{T})(x)|<2\epsilon'D\norm{\xi}_\infty\norm{\eta}_\infty \mu(K_1)\mu(K_2)\norm{\phi}^2_\infty\leq \frac{2\epsilon}{3},
\end{align*}
for all $x\notin B(e,C)$. Hence whenever $x \notin B(e,C)$,
\[
|S_{\omega_{\xi,\eta}}(\Phi(H))(x)|\leq |S_{\omega_{\xi,\eta}}(\Phi(H)-\Phi(T))(x)|+|S_{\omega_{\xi,\eta}}(\Phi(T))(x)|<\frac{\epsilon}{3}+\frac{2\epsilon}{3}=\epsilon. \qedhere
 \]
\end{proof}

We are ready to prove the main theorem of this paper. Recall that $C_b^{lu}(G)\cong C(\beta^{lu}(G))$ and $C_b^{lu}(G)/C_0(G)\cong C(\partial G)$.
\begin{thm} \label{thm:amenable at infinity and exactness}
Let $G$ be a locally compact second countable group. Then the following conditions are equivalent.
\begin{itemize}
\item[(1)] $G$ is amenable at infinity.
\item[(2)] $G$ is exact.
\item[(3)] The sequence 
\begin{align*}
0\rightarrow C_0(G)\rtimes_{L,r}G\rightarrow C(\beta^{lu}(G))\rtimes_{L,r}G \rightarrow C(\partial G)\rtimes_{L,r}G \rightarrow 0
\end{align*}
is exact.
\item[(4)] $C(\partial G)\rtimes_{L}G\cong C(\partial G)\rtimes_{L,r}G$ canonically.
\item[(5)] $C(\beta^{lu}(G))\rtimes_{L,r}G$ is nuclear.
\end{itemize}
\end{thm}
\begin{proof}
We will show that $(1)\Rightarrow (2)\Rightarrow (3)\Rightarrow (1)$, $(1)\Rightarrow(4)\Rightarrow (3)$ and $(1)\Leftrightarrow (5)$.

$(1)\Rightarrow (2)$: This follows from \cite[Theorem~7.2]{MR1926869}.

$(2)\Rightarrow (3)$: This follows from the  definition of the exactness of $G$.

$(3)\Rightarrow (1)$: We will show that if $G$ is \emph{not} amenable at infinity, then the sequence in condition (3) is \emph{not} exact. It follows from \cite[Corollary~2.9]{MR3420532} and Proposition \ref{metriclattice-propertyA} that there exists a uniformly locally finite metric lattice $Z$ in $G$ without Yu's property $A$. Theorem \ref{ghost-A} and Proposition \ref{mainprop2} imply that $C_u^*(Z)$ contains a non-compact ghost $T$ and that $\Phi(T)$ is an obstruction for the exactness of the following sequence:
\begin{align*}
0\rightarrow C_0(G)\rtimes_{R,r}G\rightarrow C_b^{ru}(G)\rtimes_{R,r}G \rightarrow (C_b^{ru}(G)/C_0(G))\rtimes_{R,r} G \rightarrow 0.
\end{align*}
We claim that this sequence is exact if and only if the sequence in condition (3) is exact. Indeed, the inverse homeomorphism on $G$ induces a commutative diagram:
$$
\xymatrix{
0 \ar[r] & C_0(G)\rtimes_{R,r} G \ar[r] \ar[d]_{\cong} & C_b^{ru}(G))\rtimes_{R,r} G \ar[r] \ar[d]_{\cong} & (C_b^{ru}(G)/C_0(G))\rtimes_{R,r} G \ar[r] \ar[d]_{\cong} & 0 \\
0 \ar[r] & C_0(G)\rtimes_{L,r} G \ar[r]  &C_b^{lu}(G)\rtimes_{L,r} G \ar[r] & (C_b^{lu}(G)/C_0(G))\rtimes_{L,r} G \ar[r] & 0. \\
}
$$
The claim follows by an easy diagram chase.

$(1) \Rightarrow (4)$: This follows from Proposition \ref{amenat_infinity} and Theorem \ref{ADtop-dyn}.

$(4) \Rightarrow (3)$: Consider the following canonical diagram:
$$
\xymatrix{
0 \ar[r] & C_0(G)\rtimes_L G \ar[r] \ar[d]_{} & C(\beta^{lu}(G))\rtimes_L G \ar[r] \ar[d]_{} & C(\partial G)\rtimes_L G \ar[r] \ar[d]_{} & 0 \\
0 \ar[r] & C_0(G)\rtimes_{L,r} G \ar[r]  & C(\beta^{lu}(G))\rtimes_{L,r} G \ar[r] & C(\partial G)\rtimes_{L,r} G \ar[r] & 0. \\
}
$$
Note that the diagram commutes and the middle vertical arrow is surjective. Since the top sequence is exact and the right vertical arrow is injective, the bottom sequence is also exact by an easy diagram chase.

$(1) \Rightarrow (5)$: This follows from Proposition \ref{amenat_infinity} and Theorem \ref{ADtop-dyn}.

$(5) \Rightarrow (1)$: By assumption, the two $*$-isomorphic crossed products 
\[
C_b^{ru}(G)\rtimes_{R,r}G\cong C(\beta^{lu}(G))\rtimes_{L,r}G
\]
 are nuclear. Let $Z$ be a uniformly locally finite metric lattice in $G$. Then it follows from Proposition \ref{mainprop} (1) that the identity map on $C_u^*(Z)$ factors through the nuclear $C^*$-algebra $C_b^{ru}(G)\rtimes_{R,r}G$ by c.c.p. maps. Hence $C_u^*(Z)$ is also nuclear. It follows from Theorem \ref{A_Nuclear} and Proposition \ref{metriclattice-propertyA} that $G$ has property $A$. Therefore, $G$ is amenable at infinity by \cite[Corollary~2.9]{MR3420532}. 
\end{proof}
\section{Concluding Remarks}
We conclude by connecting our work to a remarkable result of Hiroki Sako:
\begin{thm}[{{\cite[Theorem 1.1]{Sako13}}}]
Let $(Z,d)$ be a uniformly locally finite metric space. Then the following conditions are equivalent:
\begin{enumerate}
 \item The metric space $(Z,d)$ has Yu's property $A$.
 \item The uniform Roe algebra $C^*_u(Z)$ is nuclear.
 \item The uniform Roe algebra $C^*_u(Z)$ is exact.
 \item The uniform Roe algebra $C^*_u(Z)$ is locally reflexive. 
\end{enumerate}
\end{thm}
For a general $C^*$-algebra, nuclearity implies exactness, and exactness implies local reflexivity. Moreover, a $C^*$-subalgebra of a locally reflexive $C^*$-algebra is also locally reflexive (we refer to \cite[Chapter 9]{MR2391387} for more details). Combining Sako's result with Proposition \ref{mainprop} and Proposition \ref{metriclattice-propertyA}, we obtain an analogous result on locally compact second countable groups:
\begin{cor} \label{cor:combining Sakos result}
Let $G$ be a locally compact second countable group. Then the following conditions are equivalent.
\begin{itemize}
\item[(1)] The group $G$ has property $A$.
\item[(2)] $C_b^{ru}(G)\rtimes_{R,r} G$ is nuclear.
\item[(3)] $C_b^{ru}(G)\rtimes_{R,r} G$ is exact.
\item[(4)] $C_b^{ru}(G)\rtimes_{R,r} G$ is locally reflexive.
\end{itemize}
\end{cor} 

%
%
%

\bibliographystyle{plain}
\bibliography{kangbib}
\end{document}